\documentclass[10pt,reqno]{amsart}

\usepackage[cp1251]{inputenc}

\usepackage{amssymb
,epsfig
}
\usepackage{graphicx}

\newtheorem{dfn}{Definition}
\newtheorem{thm}{Theorem}
\newtheorem{prp}{Proposition}
\newtheorem{lemma}{Lemma}

\def\R{{\mathbb R}}

\def\P{{\mathbb P}}

\def\<{\langle}
\def\>{\rangle}

\DeclareMathOperator{\rank}{rank}

\begin{document}

\title{Equilibria of point charges on convex curves}

\author{G.Khimshiashvili*, G.Panina$^\dag$, D.Siersma$^\ddag$}

\address{*Ilia State University, Tbilisi, Georgia,
e-mail: khimsh@rmi.acnet.ge.
 $^\dag$ Institute for Informatics and Automation,
Saint-Petersburg State University, St. Petersburg, Russia,
e-mail:gaiane-panina@rambler.ru. $^\ddag$ University of Utrecht,
Utrecht, The Netherlands, e-mail: D.Siersma@uu.nl.}



\begin{abstract}
We study the equilibrium positions of three points on a convex curve
under influence of the Coulomb potential. We identify these
positions as \textit{orthotripods}, three points on the curve having
concurrent normals. This relates the equilibrium positions to the
caustic (evolute) of the curve. The concurrent normals can only meet
in the \textit{core} of the caustic, which is contained in the
interior of the caustic. Moreover, we give a geometric condition for
three points in equilibrium with positive charges only. For the
ellipse we show that the space of orthotripods is homeomorphic to a
2-dimensional bounded cylinder.

\end{abstract}

\maketitle

\section{Introduction}

The electrostatic (Coulomb) potential of a system of point charges
confined to a certain domain in Euclidean plane has been studied in
a number of papers in connection with various problems and models
of mathematical physics \cite{agheriva}, \cite{exner}, \cite{ross}.
Analogous problems for the gravitational potential of point masses
have also been discussed in the literature \cite{agheriva}, \cite{exner},
\cite{ross}.

Several issues considered in the mentioned papers were connected
with studying equilibria configurations of a system of point charges.
As usual, a configuration $P$ of several points on a curve $X$ is called
an equilibrium for a system of charges $Q$ if the system is `at rest' at this configuration.
In particular, equilibria of three point charges on an ellipse have
been discussed in some detail in \cite{agheriva}.

In this paper, we deal with a specific problem related to
our previous research of charged polygonal linkages \cite{kps}. Namely,
given a collection of $n\geq 3$ points $P$ on a given (fixed) closed curve $X$,
we wish to investigate if this collection of points is an equilibrium of $E_q$
for a certain system of point charges $q$. If this is possible,
any such collection of charges $q$ will be called
 {\it balancing charges} for $P$. We also want to investigate if this can be done with positive charges only. An analogous definition is meaningful and
interesting for several other potentials of point interactions, e.g., for gravitational
potential or logarithmic potential in the plane and more general central forces.

In the sequel we deal mostly with the (electrostatic) Coulomb
potential $E_q$ and the $E_q$-equilibrium problem for triples of
points on a closed curve in the plane. We reveal that this problem
is closely related to certain geometric issues concerned with the
concept of {\it caustic} (evolute) of a plane curve and present
several results along these lines. In particular, we give a
geometric criterion for balancing three points on an arbitrary
closed curve (Theorem \ref{t:3pod}) and present a condition for
balancing with positive charges (Proposition  \ref{poscharges}). We
give a description of the set of triples which can be balanced  on
an ellipse, and also a description of the set of triples on an
ellipse which can be balanced  by positive charges. Some related
results and arising research perspectives are discussed in the last
section of the paper.

An important observation is that Coulomb forces can be replaced by
Hooke forces produced by (either compressed or extended) connecting
springs, or by any other \textit{central forces}.

\section{Electrostatic equilibrium of points on closed curve}

\subsection{Condition for triples in equilibrium}
We consider a collection $P_1, \cdots ,P_n$ of  distinct points  on
a smooth curve $X$ in the plane, together with charges $q_1,\cdots,q_n$.
We do not require in this section that the curve is convex.
The Coulomb potential of these point charges is given by
$$   E_q =  - \sum_{i < j}\frac{q_iq_j}{d_{ij}},$$  where
$p_i=\overrightarrow{OP_i}$, and $d_{ij} = || p_i-p_j || . $ The
Coulomb forces between these points are given by

$$F_{ji} = \frac{q_iq_j}{d_{ij}^3}(p_i-p_j)
$$

Let $F_i = \sum_{j \ne i} F_{ji}$ be the resultant of these forces
at $P_i$. Let $T_i$ be the tangent vector to the curve $X$ at the
point $P_i$.

\begin{dfn} \label{d:station}
 A collection of points on a curve  $X$  charged by  $q=(q_1,\cdots, q_n)\neq (0,...,0)$ is called an $E_q$-equilibrium (or is $E_q$-balanced) if, at every point $P_i$
  of the collection, the resultant of the forces is orthogonal to $T_i$:

  $$ \langle F_i, T_i\rangle =
\sum_{j \ne i} \frac{q_iq_j}{d_{ij}^3}\langle p_i-p_j, T_i\rangle =
0 \; \; \forall i. $$
  In this situation we say that the charges $q$ are balancing for $P_1, \cdots ,P_n$.
\end{dfn}
Notice that $E_q$-equilibria correspond to the critical/stationary points of the potential $E_q$.

Whenever $q_i=0$ for one of the charges, the system reduces to a
system with one point less (the removed point can be on an arbitrary
place). If none of the charges is zero then the equilibrium
condition implies a system of linear equations for the values of
balancing charges.

In the special case of two points we have two equations.
Non-zero solutions $q=(q_1,q_2)$ only occur if $P_1P_2$ is
orthogonal to both $T_1$ and $T_2$. So $P_1P_2$ is a double normal
of the curve.

The main situation of our study is three points on a curve. In this
case we have the matrix equation:

\[
\begin{pmatrix}

 0 & a_{12}     & a_{13} \\
a_{21} &  0 & a_{23} \\
a_{31} & a_{32} & 0
    \end{pmatrix}
\begin{pmatrix}
q_1\\ q_2 \\ q_3
\end{pmatrix}
= 0 \; , \; \mbox{where} \; \; a_{ij}= \frac{\langle p_i-p_j ,
T_i\rangle}{d_{ij}^3} \; .
\]

If the rank of this system is 3 then its solution is only
the triple $(0,0,0)$ so a genuine (non-trivial)
equilibrium is impossible. Thus non-trivial stationary charges
may only exist if the rank of this system does not exceed two.

\begin{dfn}
Points $P_1, P_2, P_3$ satisfy  the corank 1 condition if the rank
of the matrix
$$(a_{ij} )= \Big(\frac{\langle p_i-p_j, T_i\rangle}{d_{ij}^3}\Big)$$ is less than or equal to $ 2$.

\end{dfn}

If the rank of the matrix
$(a_{ij} )$ is equal to 2, then the matrix equation has a
one-dimensional solution space and therefore defines a unique point
$[q_1:q_2:q_3]$ in $\mathbb{P}^2$. In case where the rank of this
matrix is 1 it follows that the points $P_1, P_2, P_3$ are collinear
and at least one of $P_iP_j$ is a double normal. This case cannot
occur if the curve is convex.

\begin{prp}\label{p:rank2}
 Points $P_1, P_2, P_3$ satisfy the corank 1
condition if and only if the three normals at these points are
concurrent, that is, have a common point.
\end{prp}

\vspace{-0.5cm}

\begin{proof}
\begin{figure}[htbp]
    \centering
        \includegraphics[width=10 cm]{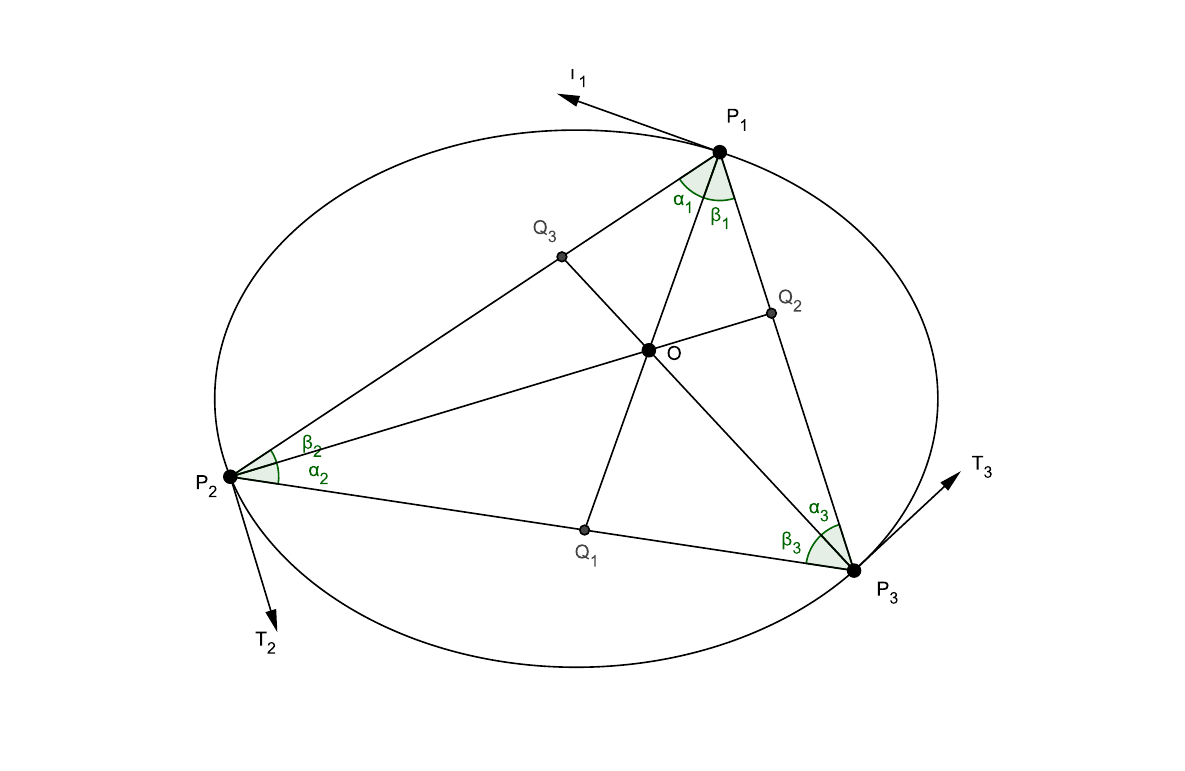}

            \vspace{-1cm}
    \caption{Orthotripods and Ceva configuration}
    \label{fig:ceva}
\end{figure}

Introducing the
angles $\alpha_i = \angle{P_{i+1},P_iQ_i}$, $\beta_i = \angle{Q_iP_iP_{i-1}}$ (with mod 3 convention) and unit tangent vectors $T_i$, we notice that
$$a_{12} = \frac{ \langle p_1 - p_2 , T_1\rangle }{d_{12}^3}= - \frac{ \sin \alpha_1}{d_{12}^2}, \; \; \; \; \; \;
a_{13} = \frac{\langle p_1 - p_3  , T_1 \rangle}{d_{13}^3} =  \frac{\sin \beta_1 }{d_{13}^2} , \; \; \; \mbox{etc.}$$

We 
conclude that $\rank (a_{ij})$ does not exceed two if and only if
$$a_{12}a_{23}a_{31}
+ a_{13}a_{21}a_{32} = 0.$$
The latter condition is equivalent to
$$\sin \alpha_1 \sin \alpha_2
\sin \alpha_3 = \sin \beta_1 \sin \beta_2 \sin \beta_3.$$ By Ceva
theorem the lines $P_1Q_1, P_2,Q_2, P_3Q_3$ are concurrent if and
only if the signed lengths of segments satisfy the following
relation.
$$ \frac{P_2Q_1}{Q_1P_3} \cdot \frac{P_3Q_2}{Q_2P_1} \cdot \frac{P_1Q_3}{Q_3P_2} = 1,$$
which is in our case  equivalent to
$$\frac{\sin \alpha_1}{\sin \beta_1} \cdot \frac{\sin \alpha_2}{\sin \beta_2} \cdot \frac{\sin \alpha_3}{\sin \beta_3} = 1.$$
\end{proof}

\bigskip

\textbf{Remark.} The same proposition holds for any  \textit{central
forces}, that is, for all the forces that are given by: 

$$F_{ji} = \frac{q_iq_j}{f_{ij}}(p_i-p_j) \; \;
\mbox{where} \; \; f_{ij} = f_{ji} \; \mbox{depends only on} \; \;  d_{ij}.  $$

Examples of central forces are Coulomb forces, Hooke forces, and
logarithmic forces. The Ceva equation does not depend on the
concrete central force, since the denominators of $a_{ij}$ cancel in
the condition of Proposition \ref{p:rank2}.
 We can as well work with $a_{ij}= \langle p_i-p_j , T_i\rangle$. Proposition \ref{poscharges} below also holds in this more general case.
\bigskip

\textbf{Remark.} The claim of Proposition \ref{p:rank2} involves
only  the points $P_i$ and the tangent vectors $T_i$ and no other
data of the curve.

\bigskip

\begin{dfn} \label{d:ortho3pod}
An orthotripod on a smooth closed curve $X$ is defined as an
\textbf{unordered triple} of distinct points such that the normals
to $X$ at these points are concurrent. The common point of these
three normals is called the orthotricentre.
\end{dfn}

Notice an analogy of this result with a well-known description
of three forces in equilibrium.

\bigskip

\textbf{Remark.}
The name ''tripod'' was used  by S. Tabachnikov in \cite{tabach}  for three concurrent normals to the curve making angles of 120 degrees (Steiner property).\\
Orthotripods occur also in \cite{bramourr} and \cite{Imm2}, where
the authors give a necessary condition for immobilization  of convex
curves by three points. In those papers they are called normal
triples.

\bigskip

\begin{thm}\label{t:3pod}
Given 3 points on a curve $X$,  there exist charges $(q_1:q_2:q_3)$ such the points are in $E_q$-equilibrium if and only if they form an orthotripod on $X$.
\end{thm}

\begin{proof}
Follows from Proposition \ref{p:rank2}.
\end{proof}

In such a case the values of balancing charges can be explicitly calculated (cf. Section \ref{ss:computingcharges}).
For the moment we search for positive balancing charges.

\begin{prp}\label{poscharges}
 If $P_1, P_2, P_3$ satisfy the corank 1 condition,
then there exist positive balancing charges if and only if the three
normals at points $P_i$ enter the interior of triangle $\triangle P_1P_2P_3$.
\end{prp}

\begin{proof}
 For a solution $(q_1,q_2,q_3)$ with all $q_i > 0$, we have
\begin{center}
 $\mbox{sign} \; a_{12} = - \mbox{sign} \; a_{13}$ \\
$\mbox{sign} \; a_{21} = - \mbox{sign} \; a_{23}$ \\
$\mbox{sign} \; a_{31} = - \mbox{sign} \; a_{32}$. \\
\end{center}

\begin{figure}[htbp]
\vspace{-1cm}
          \includegraphics[width=14cm]{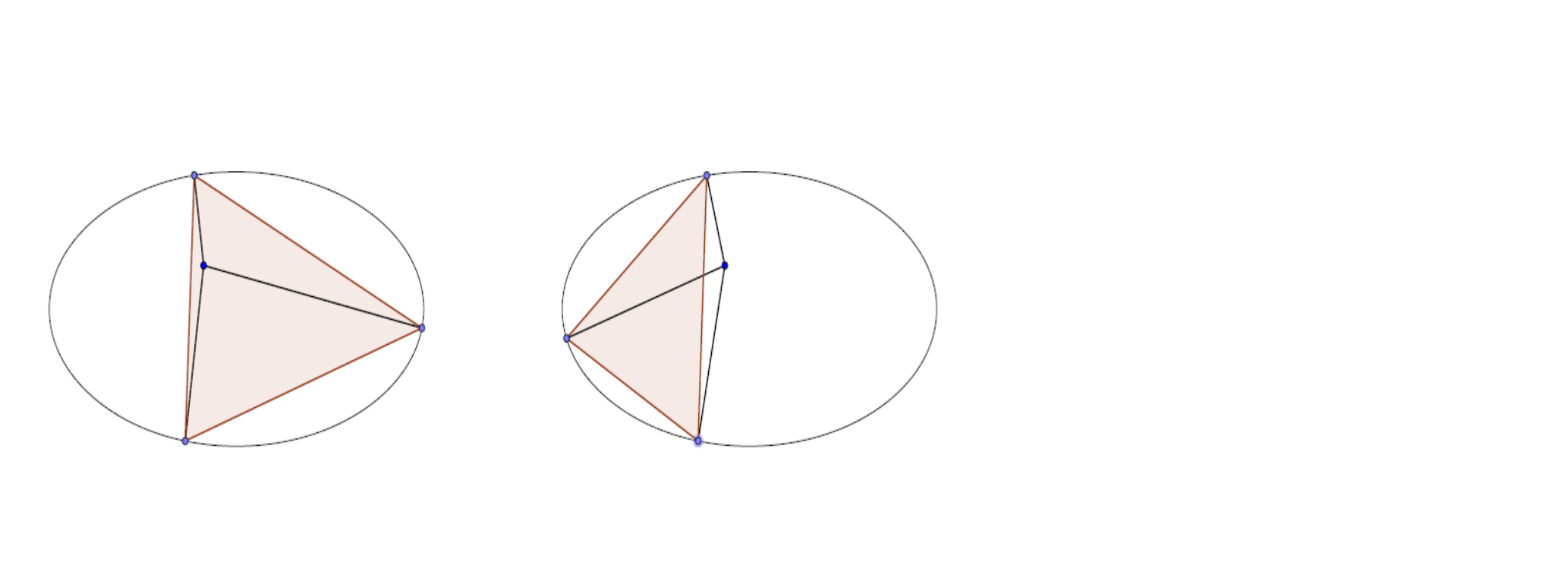}
		
		\vspace{-1cm}		
    \caption{Orthotripods in equilibrium, with positive charges (left) and different signs (right)}
    \label{fig:TwoO3ples}
\end{figure}

Since (up to a positive factor) $a_{12} =  \langle p_1 - p_2 ,
T_1\rangle$ and $a_{13} =  \langle p_1 - p_3  , T_1\rangle$ it
follows that the edges $P_1P_2$ and $P_1P_3$ are on different sides
of the normal at $P_1$ .
\end{proof}

\section{Caustics}

The above results indicate an interesting connection with some classical
issues of differential geometry which we outline below. To this end we will
need some definitions and auxiliary results to be used in the sequel.

\begin{dfn} Caustic $C(X)$ of a regular closed curve $X$ is defined as the set of curvature
centers at all points of $X$ (also known as evolute or focal set).
\end{dfn}

 It is known that caustic may be equivalently defined as the envelope of normals to $X$
or the set of singular points of the wavefront of parallel curves
(on a given distance of $X$)  \cite{brgi}. For a generic curve $X$,
its caustic is a piecewise smooth curve with cusps and normal
crossings as singularities (see, e.g. \cite{avg}). Notice also that
all parallel curves  have the same caustic as $X$. They share with
$X$ all properties, which can extracted from the caustic.

\begin{figure}[htbp]
    \centering
        \includegraphics[width=11 cm]{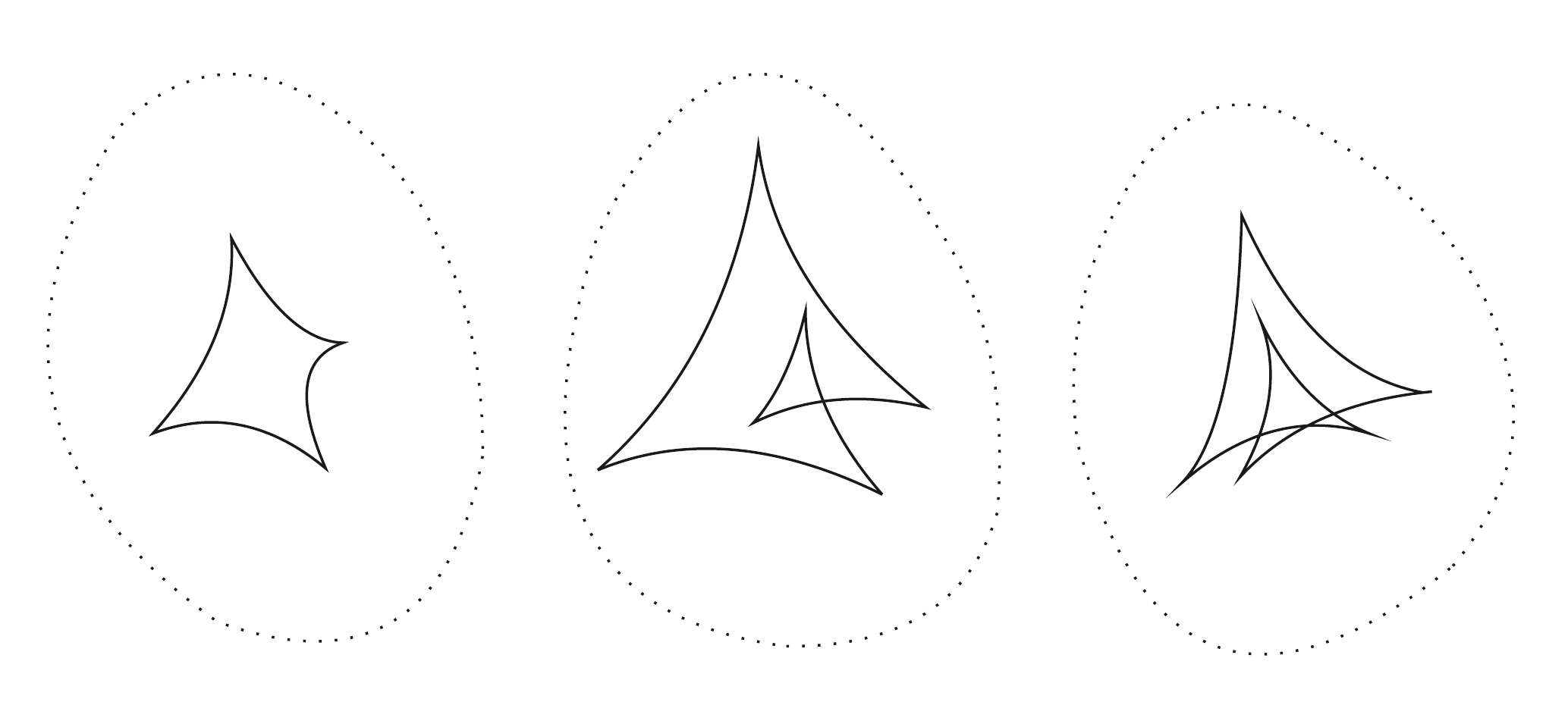}

    \caption{Some curves and their caustics}
    \label{fig:convexcause3}
\end{figure}

Our next aim is to describe in some detail the relation between
caustics and orthotripods. In order to  explore further aspects of
this connection we need some more definitions and notation. We
assume that the curve $X$ is parameterized by its arclength and
write $X(s)$. Denote by $T(s)$ the unit tangent vector and by $N(s)$
the unit normal vector such that the pair $ T(s), N(s) $ is
positively oriented. Let the function $D: \R^2 \times X \rightarrow
\R$ be given by
$$D(x,y,s) = | (x,y) - X(s)|^2. $$

The critical set  $M = \{(x,y,s): \partial_x D = \partial_y D =0\}$
 of the mapping $D$ is
generically smooth and the projection $\Pi: M \rightarrow \R^2$ is a mapping
between smooth two-dimensional manifolds. The image of the critical set $M$
is exactly the caustic $C(X)$ defined above.

For a point $Q\in \R^2$, a pair $(Q,P)$  belongs to the preimage
$\Pi^{-1}(Q)$  iff  the normal to the curve $X$ through $P$ contains
$Q$. Thus the number of points in the preimage of $Q$ is equal to
the number of normals to the curve $X$ emerging  from $Q$. This
number is even if $Q$ does not belong to caustic. Notice that if we
count the preimages with signs then we get the degree of $\Pi$ which
vanishes.

For further use we mention a few other properties of $\Pi$. In particular, the map
$\Pi$ at a singular point $(Q,P)$ is a fold map if $\kappa'(P) \neq 0$, and it is a cusp
map if $\kappa'(P)=0$ and $\kappa'(P)$ has a simple zero at $P$, where $\kappa(P)$ is the curvature at the point $P$.\\

In the rest of this section we assume that $X$ is a convex curve and that $\kappa(P)
\neq 0$ (no flat points). Then the caustic $C(X)$ divides the plane
in closed domains. For each of the points $Q$ of the plane let us
denote by $n(Q)$ the number of normals to the curve $X$  emerging
from  $Q$. This number is constant on each of the (interiors of the)
domains. Points from the non-compact domain have only (exactly) two
normals to $X$.

\begin{dfn}
The union of the compact domains is called the interior region of
the caustic. The  core of the caustic is defined as the union of the
closed domains where the points have at least four emerging normals.

\end{dfn}
The co-orientation of the caustic along the fold-lines (in the
direction of two extra normals) together with the orientation of the
plane $\R^2$ yield an orientation of the caustic. Let for any point
$Q$ in the complement of the caustic  $i(Q)$ be the degree of $Q$
with respect to the caustic.

\begin{lemma}\label{core}In the above notation, we have:
\begin{enumerate}
    \item $n(Q) = 2 i(Q) + 2 $.
    \item The core of the caustic is equal to the closure of the set of points $Q$ with $i(Q) \ne 0$.
\end{enumerate}
\end{lemma}
\begin{proof}
The index of point with respect to a closed curve can be computed by
taking a generic half-ray starting at that point and counting the
intersection points with the curve with a sign. Outside a compact
area $i(Q) = 0$, while $n(Q) = 2$. For each intersection point of
the ray with the caustic, the number of normals changes by
 $+2$  (respectively,  $-2$) according to the change $+1$ (respectively,  $-1$) of the index.
\end{proof}

From the above discussion and Lemma \ref{core} we get a characterization of the set of orthotricenters.

\begin{prp}\label{cau3pod}
 The set of orthotricenters is a compact subset of the interior region of the caustic and coincides with the core of the caustic.
\end{prp}

\bigskip

\textbf{Remark.} In the paper of Gounai and Umehara \cite{gounai} is
shown that there are exactly 3 types of caustics with 4 cusps only
(Figure \ref{fig:4cusps}). Inspection of these cases shows that the
core can be different from the internal region.  Lemma \ref{core}
implies that  there are only 2 normals in the ''holes''.

\begin{figure}[htbp]
    \centering
        \includegraphics[width=15 cm]{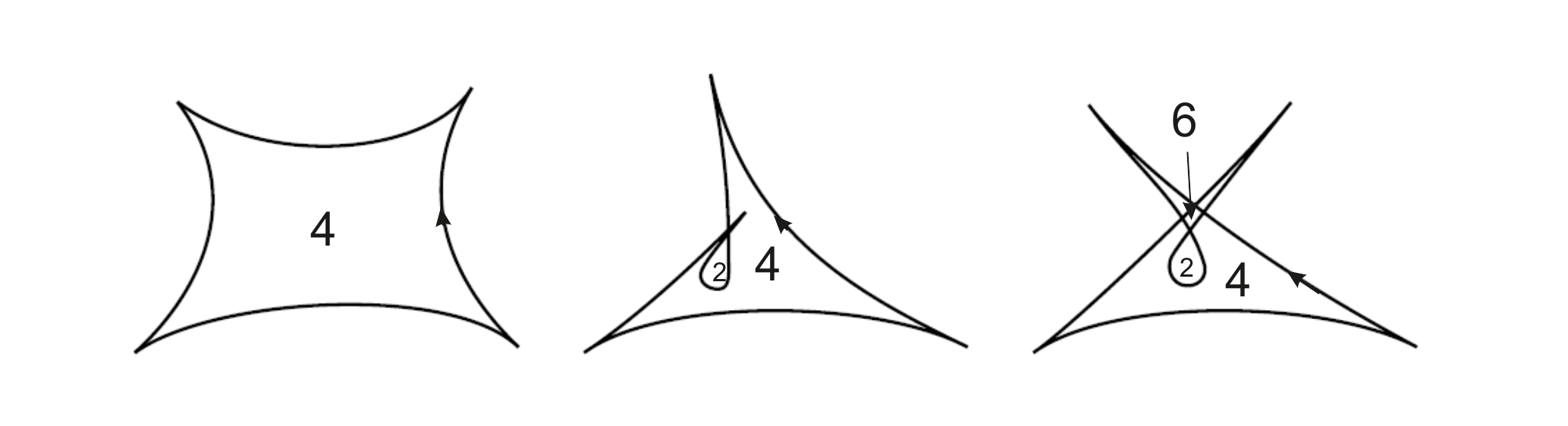}
    \caption{The three 4-vertex caustics  with  orientation and the values $n(Q)$.}
    \label{fig:4cusps}
\end{figure}
\bigskip

Hence for studying orthotripods
 we only need to consider orthotricenters in the core
of caustic.

\vspace{0.5cm}

Next we discuss the equilibria with positive charges. Double normals
to  $X$ provide an extra structure. It is known that each curve has
at least two double normals. There exist curves with infinitely many
double normals (circle  other curves of constant width) but they are
highly non-generic, so here we assume that the number of double
normals is finite.

\begin{lemma}
Assume that an orthotripod changes continuously so that the
orthotricenter  crosses none of double normals. Then the signs of
the balancing charges do not change.
\end{lemma}
\begin{proof}  Away from double normals, balancing charges depend
continuously on the orthotripod. If one of the charges changes the
sign, then by  Proposition \ref{poscharges}
 the orthotricenter crosses a double normal. \end{proof}

 Double normals yield a
further partition of the core of the caustic into smaller (open)
regions.
In each of the regions  the three charges form a non-zero triple
with constant signs.

\section{Computing balancing charges} \label{ss:computingcharges}
\subsection{Balancing charges} \label{s:staioncharges}
In case the points $P_1,P_2,P_3$ satisfy the corank 1 condition (i.e. their normals form an orthotripod) we can search for the balancing charges $(q_1,q_2,q_3)$.
 If the rank is exactly two they will define a one-dimensional subspace of $\R^3$.
Since only the ratio is important, we have a point in $\P^2$.
Straightforward computations show that a solution is given by:
$$ [ q_1 : q_2: q_3 ] = [ - \frac{a_{23}}{a_{21}} : - \frac{a_{13}}{a_{12}} : 1] $$
We use the following notations:
$$ A_1  = \frac{a_{13}}{a_{12}} = - \; \frac{d_{12}^2}{d_{13}^2} \frac {\sin \beta_1}{\sin \alpha_1} \; , \;
A_2 = \frac{a_{21}}{a_{23}} = - \; \frac{d_{23}^2}{d_{12}^2} \frac {\sin \beta_2}{\sin \alpha_2} \;, \;
A_3 = \frac{a_{32}}{a_{31}} = - \; \frac{d_{13}^2}{d_{23}^2} \frac {\sin \beta_3}{\sin \alpha_3}.
 $$

Keeping in mind the corank 1 condition $A_1A_2A_3= - 1$ we can write the solution in one of the following forms:
$$ [q_1 : q_2: q_3] =[ A_1A_3 : - A_1 : 1] =[ - A_3 : 1 : A_2A_3] =[ 1: A_1A_2 : - A_2 ] .$$

This formula is for Coulomb forces, in other cases there are similar formulas with different coefficients.

\subsection{Three points on a line}

Let $P_1$, $P_2$ and $P_3$ be located on a line in this order with
distances $d_{12}= a$, $d_{23} = b$, then these points are in
equilibrium if and only if $$[ q_1 : q_2 : q_3 ] = [ \frac{1}{b^2} :
- \frac{1}{(a+b)^2} : \frac{1}{a^2} ] ,$$ that is, every 3 points
can be balanced, but never with positive charges.

\subsection{Three charges on a circle}

Our original paradigm is given by the case of a circle, considered, e.g., in \cite{cohn}. The following can be shown by direct calculations
or as corollaries of Theorem \ref{t:3pod}  and Proposition \ref{poscharges}.
\begin{enumerate}
    \item  Any 3 points on a circle can be  balanced.
    \item Three points on a circle can be balanced with positive charges iff they form an acute triangle.
\end{enumerate}

\subsection{Three charges on a parabola, near to the vertex}
We use a parabola $y= c x^2$ and points $P_1 = (-t, ct^2),
P_2=(0,0), P_3 = (t, ct^2)$. Direct calculations shows that these
points are in equilibrium iff $$[ q_1 : q_2 : q_3 ] = [ 1 : -
\frac{1}{4}  \frac{(1+ 2c^2t^2)^{\frac{3}{2}}}{1+c^2t^2} : 1 ].$$
Notice that if $t \rightarrow 0$ we get $[1: - \frac{1}{4} ; 1 ]$,
which is in accordance with
  with  3 points on the line. In fact
near to the vertex (on any curve)  we cannot have an equilibrium
with only positive charges.

\subsection{Three charges on an ellipse}

Several results in this case have been obtained in \cite{agheriva}.
However some formulations in \cite{agheriva} are not completely rigorous
and one of the aims of our research was to clarify several issues
discussed in \cite{agheriva} in full generality.
Due to our Theorem \ref{t:3pod} and  Propositions \ref{poscharges} and \ref{cau3pod} we obtain mathematically rigorous statements in the case of three points on the ellipse:
\begin{enumerate}
    \item Equilibria exist as soon as the normals are concurrent.
    \item For each point in the core of the caustic there are exactly four concurrent normals, which give rise to four orthotripods. On the caustic curve itself we have only one orthotripod.
    \item In general two out of four of these orthotripods correspond to equilibria with positive charges. On the double normals one has at least one charge zero.
\end{enumerate}

\section{Topology of orthotripods on an ellipse}

If an ellipse is a circle, then the space
of orthotripods is the symmetric cube of $S^1$  with the fat diagonal deleted.

Assume that the ellipse is not a circle. Then it has a caustic with
4 cusps and no double points. We define $\overline{T}$ as the
closure of the space of orthotripods $T$  taken in the space of all
unordered triples of points on the curve (the symmetric cube of
$S^1$). Let $Y$ be the core of the caustic. In the case of the
ellipse this is a topological disc bounded by 4 intervals, meeting
in 4 cusps.
 There is a projection
$$\overline{T} \rightarrow Y,$$
sending each orthotripod to its orthotricenter.

Over the interior points of $Y$ we have a 4-sheeted covering, and
$\overline{T}$  is a patch of  these 4 sheets along some of the
boundary edges. The rules for that are related to the fold
singularities along the edges. Their fold lines separate exactly two
of the sheets.

If $Q$ is an interior point of $Y$, we just number the perpendiculars
counterclockwise, and then extend the numbering to all of $Y$.
    Sheets are now denoted (with the numbers of the perpendiculars) by $Y_{123}, Y_{234}, Y_{341}, Y_{412}$ (no ordering involved).
    Next we label the cusp and the edges of the caustic with the labels of the coinciding normals:
    \begin{enumerate}
        \item $(1,2)$ means that on that edge the normals $1$ and $2$ coincide;
        \item the cusp $(1,2,3)$ has the property that the normals with labels $1,2,$ and $3$ coincide.
    \end{enumerate}

The gluing rule for edges   is now as follows. The edge $(1,2)$ is
used to glue the corresponding edges of $Y_{123}$ and  $Y_{412}$. On
the sheets $Y_{234}, Y_{341}$  there is no identification, the edge
corresponding to $(1,2)$ will survive as boundary. For the other
edges we have similar behavior. Around cusp points there are 3
sheets involved and the folding changes from one pair to some other
pair of sheets. See the Figures \ref{fig:patching1} and
\ref{fig:patching2} for the details about the gluing.

\begin{figure}[htbp]
\centering
        \includegraphics[width= 13 cm]{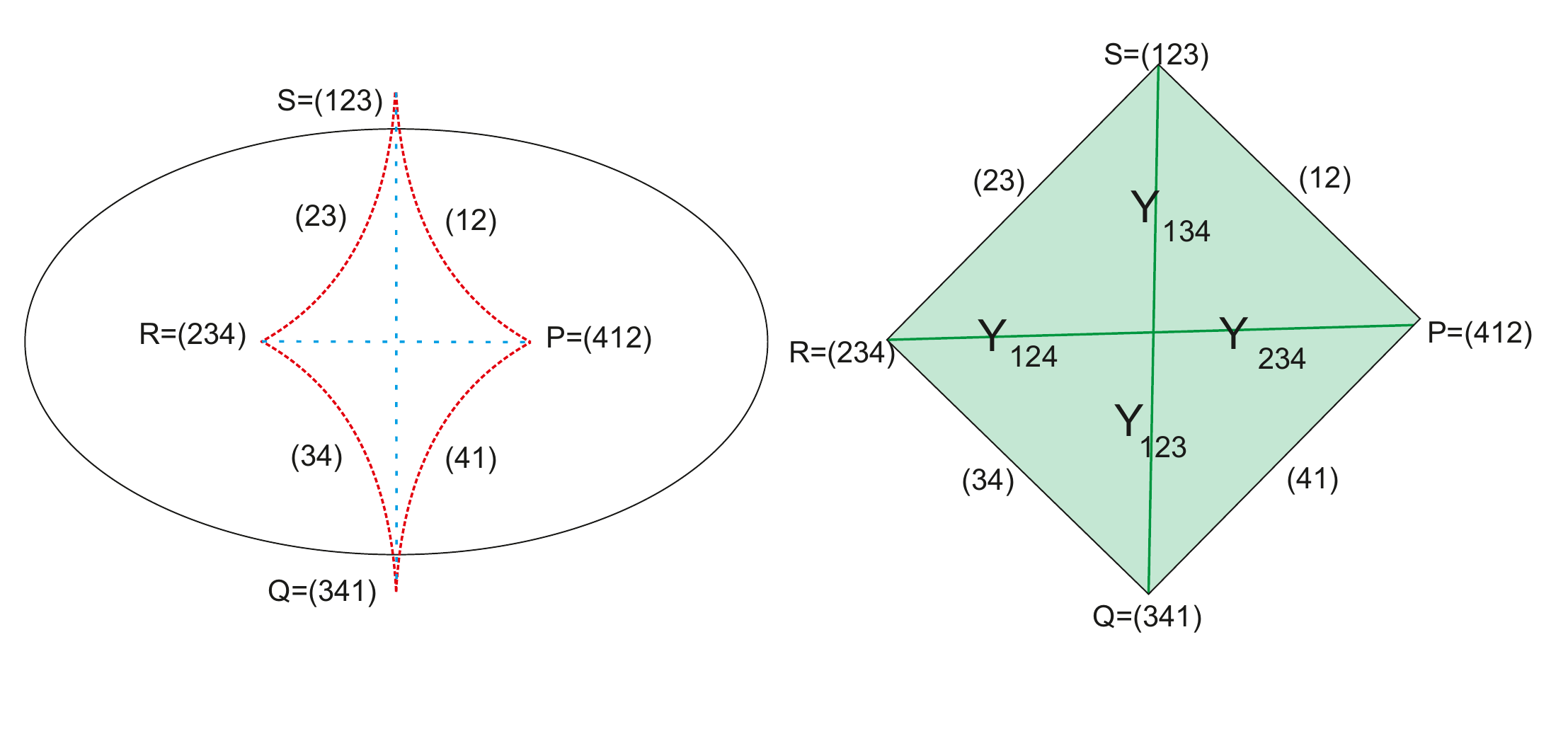}
\vspace{-0.5cm}
    \caption{The ellipse and its caustic (left); Indications of regions with positive charges (right)}
    \label{fig:patching1}
\end{figure}

\begin{figure}[htbp]
\centering
        \includegraphics[width= 10.5 cm]{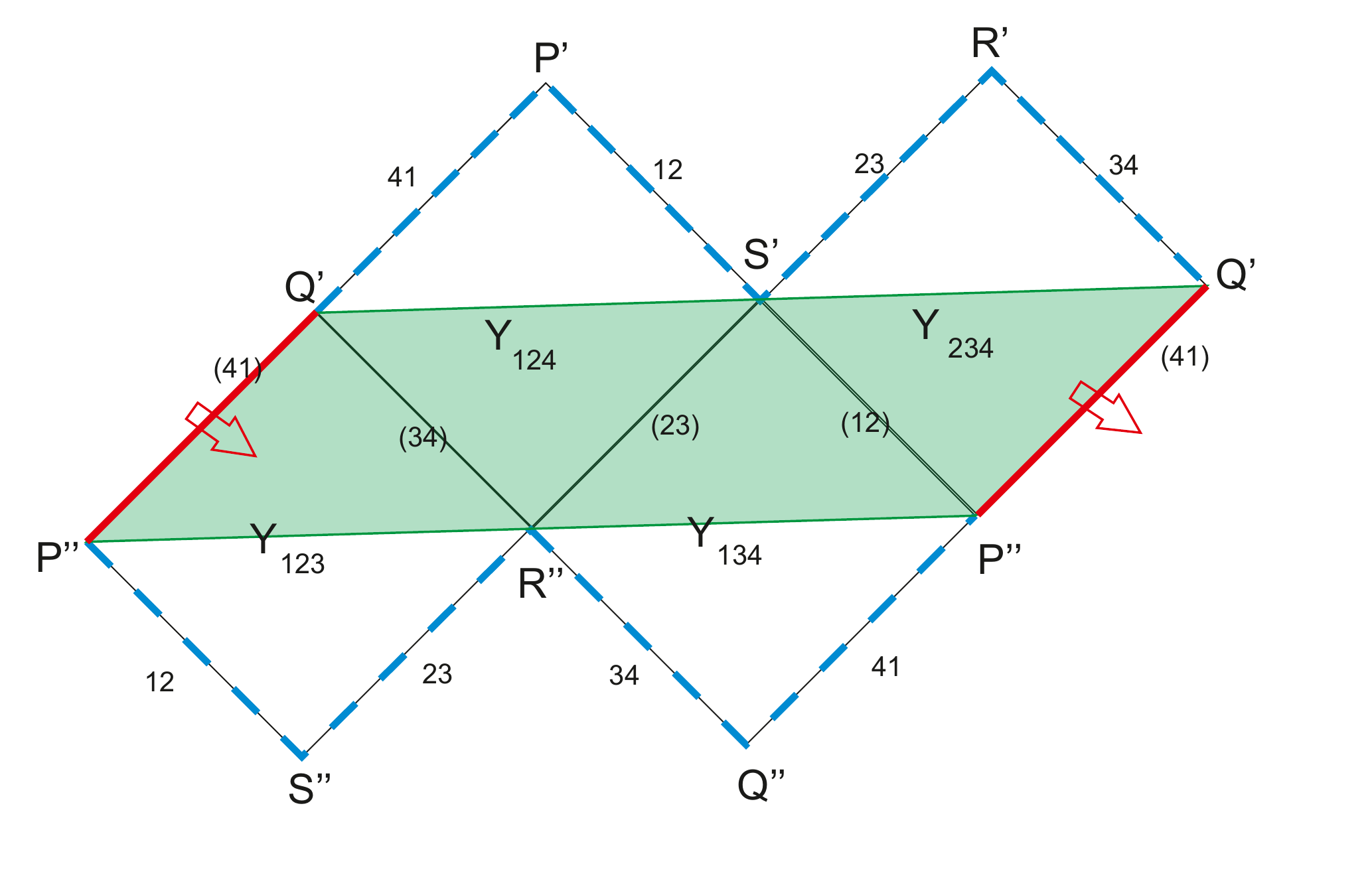}
    \caption{The cylinder of orthotripods on the ellipse. The blue dashed lines represent the boundary of the cylinder. The two thick red segments patch together.
    The shaded area represents the triples that can be balanced by positive charges.}
    \label{fig:patching2}
\end{figure}
We are now in a position to establish the concluding result.

\begin{prp}\label{p:top3Pellipse}
Assume that we have an ellipse, which is not a circle.
\begin{enumerate}
    \item  The closure of the
space of orthotripods is homeomorphic to a cylinder (with two
boundary circles).
    \item  The space of orthotripods which can be
balanced by positive charges only is homeomorphic to  a cylinder.
\end{enumerate}
\end{prp}
\begin{proof}
As we explained above, the  closure of the space of orthotripods is a patch of four copies
of the core of the caustic.  This gives us a cylinder. On each of the copies we shade green the part that corresponds to positive charges,
which gives us a smaller cylinder (see Figure \ref{fig:patching2}).

\end{proof}

\section{Concluding remarks and questions}

As we can see from Figure \ref{fig:convexcause3}, the combinatorics
of a caustics can vary a lot.  An intriguing question is to
characterize the space of orthotripods in full generality, not just
for ellipse (for which we have the simplest possible caustic).
Evidently, similar patching rules hold for any caustic, but can lead
to a more complicated topology of the space of orthotripods. Natural
questions are: is this space a surface with boundary? What is its
Euler characteristic? What is its genus, number of boundary
components?

Another issue which remains untouched in this paper is the question
of stability of an equilibrium. It seems that here more delicate
characteristics of the curve are involved: not just  the
combinatorics of tangent lines and normals, but also the curvature
of the curve.

One more  interesting question could be for fixed charges
$[q_1:q_2:q_3]$, to relate  the number  of orthotripods balanced by
these charges to the properties of the curve.

\bigskip

\textbf{Acknowledgements.} The present paper was written during a
``Research in Pairs'' session in CIRM (Luminy) in January of 2015. The
authors acknowledge the hospitality and excellent working conditions
at CIRM

\end{document}